\documentclass[runningheads]{llncs}

\usepackage[T1]{fontenc}
\usepackage{graphicx,subcaption}

\usepackage{xcolor}
\usepackage{amsmath,amssymb,mathtools}
\mathtoolsset{showonlyrefs}
\usepackage[english,colorlinks=true,linkcolor=blue,citecolor=green!50!black,bookmarks=true]{hyperref}
\usepackage{pdfprivacy}
\usepackage{orcidlink}

\newcommand{\N}{\ensuremath{\mathbb{N}}}
\newcommand{\NN}{\ensuremath{\mathbb{N}_0}}

\renewcommand{\S}{\ensuremath{\mathbb{S}}}
\newcommand{\Z}{\ensuremath{\mathbb{Z}}}

\newcommand{\R}{\ensuremath{\mathbb{R}}}
\newcommand{\C}{\ensuremath{\mathbb{C}}}
\newcommand{\e}{\mathrm{e}}

\renewcommand{\i}{\mathrm{i}}
\renewcommand{\d}{\, \mathrm{d} }
\newcommand{\zb}[1]{\ensuremath{\boldsymbol{#1}}}
\newcommand{\norm}[1]{\left\lVert #1%\smash{#1} 
  \right\rVert}
\newcommand{\abs}[1]{\left\vert#1\right\vert}
\newcommand{\spr}[1]{\left\langle\,#1\,\right\rangle}
\newcommand{\kl}[1]{\left(#1\right)}
\newcommand{\Kl}[1]{\left\{#1\right\}}

\newcommand{\yd}{y^{\delta}}

\newcommand{\fkt}{\tilde{f}_k}
\newcommand{\ekt}{\tilde{e}_k}

\newcommand{\ak}{\alpha_k}

\newcommand{\ADD}{A^\ddagger}

\spnewtheorem{assumption}{Assumption}{\bfseries}{\itshape}

\begin{document}

% Paper title
\title{A Frame Decomposition of the Funk-Radon Transform 
}

% Running paper title
\titlerunning{A Frame Decomposition of the Funk-Radon Transform}

\author{
Michael Quellmalz \inst{1} \orcidlink{0000-0001-6206-5705} \and
Lukas Weissinger \inst{2} \orcidlink{0000-0003-1032-8774} \and
Simon Hubmer \inst{2} \orcidlink{0000-0002-8494-5188} \and
Paul D.\ Erchinger\inst{1}}

\authorrunning{M.~Quellmalz, L.~Weissinger, S.~Hubmer, P.~Erchinger}

\institute{Technische Universit\"at Berlin, Institute of Mathematics, MA 4-3,
Stra{\ss}e des 17.\ Juni 136, D-10623 Berlin, Germany\\
\email{quellmalz@math.tu-berlin.de}, \email{erchinger@campus.tu-berlin.de}
\\
\url{www.tu.berlin/imageanalysis} 
\and
Johann Radon Institute Linz, Altenbergerstraße 69, A-4040 Linz, Austria\\
\email{lukas.weissinger@ricam.oeaw.ac.at, simon.hubmer@ricam.oeaw.ac.at}
\\
\url{www.ricam.oeaw.ac.at}}

\maketitle  

% % % % % % %
% Abstract  %
% % % % % % %
\begin{abstract}
The Funk-Radon transform assigns to a function defined on the unit sphere its integrals along all great circles of the sphere. 
In this paper, we consider a frame decomposition of the Funk-Radon transform, which is a flexible alternative to the singular value decomposition. 
In particular, we construct a novel frame decomposition based on trigonometric polynomials and show its application for the inversion of the Funk-Radon transform. Our theoretical findings are verified by numerical experiments, which also incorporate a regularization scheme.

\end{abstract}

% % % % % % % % % % % % % %
% Section - Introduction  %
% % % % % % % % % % % % % %
\section{Introduction}

The \emph{Funk-Radon transform} assigns to a function $f\colon \S^2\to\C$ defined on the two-dimensional \emph{unit sphere} $\S^2 \coloneqq \{\boldsymbol\xi\in\R^3: \norm{\zb\xi}=1 \}$ its integrals along all great circles of the sphere, i.e.,
\begin{equation} \label{eq:fr}
  Rf(\zb\xi)
  \coloneqq
  \frac{1}{2\pi}
  \int_{\zb\xi^\top\zb\eta=0} f(\zb\eta) \d\zb\eta \,,
  \qquad \forall\, \zb\xi\in\S^{2} \,,
\end{equation}
where $\mathrm d\zb\eta$ denotes the arclength on the great circle perpendicular to $\zb\xi$.
Tracing back to works of Funk \cite{Fun13} and Minkowski \cite{Min05} in the early twentieth century,
it is also known as Funk transform, Minkowski-Funk transform or spherical Radon transform.
It has found applications in diffusion MRI \cite{RaTiWhWe22,Tuc04}, radar imaging \cite{YaYa11}, Compton camera imaging \cite{Ter23}, photoacoustic tomography \cite{HrMoSt16}, and geometric tomography \cite[Chap.\ 4]{Gar06}.
Besides analytic inversion formulas, e.g., \cite{BaEaGoMa03,Fun13,Hel11,Kaz18},
the numerical reconstruction of functions given its Funk-Radon transform can be done using mollifier methods \cite{LoRiSpSp11,RiSp13}, the eigenvalue decomposition \cite{HiQu15}, or discretization on the cubed sphere \cite{Bel22}.
Generalizations have been developed for various non-central sections of the sphere \cite{AgRu19,Qu17,Rub22,Sal17} or for derivatives \cite{Kaz18,QuHiLo18}. 

In this paper, we are interested in \emph{frame decompositions} (FDs) of the Funk-Radon transform. Originally developed in the framework of wavelet-vaguelette decompositions \cite{Abramovich_Silverman_1998,Dicken_Maass_1996,Donoho_1995,Frikel_2013,Frikel_Haltmeier_2018,Kudryavtsev_Shestakov_2019,Lee_1997} and then extended to biorthogonal curvelet and shearlet decompositions \cite{Candes_Donoho_2002,Colonna_Easley_Guo_Labate_2010}, FDs are generalizations of the singular value decomposition (SVD) \cite{Ebner_Frikel_Lorenz_Schwab_Haltmeier_2023,Frikel_Haltmeier_2020,Hubmer_Ramlau_2020,Hubmer_Ramlau_2021_01,Hubmer_Ramlau_Weissinger_2022,Weissinger_2021}. In particular, they allow SVD-like decompositions of bounded linear operators also in those cases when the SVD itself is either unknown, its computation is infeasible, or its structure is unfavourable. More precisely, given a bounded linear operator $A\colon X \to Y$ between real or complex Hilbert spaces $X$ and $Y$, an FD of $A$ is a decomposition of the form
    \begin{equation}\label{frame_dec_A}
		A x = \sum_{k=1}^\infty \sigma_k \spr{x,e_k}_X \fkt  \,,
		\qquad \forall \, x \in X \,.
	\end{equation}
Here, the sets $\Kl{e_k}_{k\in\N}$ and $\Kl{f_k}_{k\in\N}$ form frames over $X$ and $Y$, respectively, and $\{\fkt\}_{k\in\N}$ denotes the dual frame of the frame $\Kl{f_k}_{k\in\N}$; see Section~\ref{sect_background} below. The main requirement on $e_k$ and $f_k$ is that they satisfy the quasi-singular relation
    \begin{equation}\label{cond_frames_connected}
        \overline{\sigma_k} \, e_k =  \, A^*f_k \,,
        \qquad
        \forall \,k \in \N \,,
    \end{equation}
where $\overline{\sigma_k}$ denotes the complex conjugate of the coefficient $\sigma_k \in \C$ and $A^*$ the adjoint of $A$. Using the FD \eqref{frame_dec_A} it is possible to compute (approximate) solutions of the linear equation $Ax = y$, and to develop filter-based regularization schemes as for the SVD \cite{Ebner_Frikel_Lorenz_Schwab_Haltmeier_2023,Hubmer_Ramlau_2021_01,Hubmer_Ramlau_Weissinger_2022}. However, the question remains whether frames satisfying \eqref{cond_frames_connected} can be found. While this is possible by geometric considerations for some examples \cite{Donoho_1995,Frikel_2013,Frikel_Haltmeier_2018,Hubmer_Ramlau_2020}, an explicit construction ``recipe'' exists in case that $A$ satisfies the stability condition
    \begin{equation}\label{cond_A_stability}
	    c_1 \norm{x}_X \leq \norm{Ax}_Z \leq c_2 \norm{x}_X \,,
	    \qquad
	    \forall \, x \in X \,,
	\end{equation}	
for some constants $c_1,c_2 > 0$ and a Hilbert space $Z \subseteq Y$. In this case, one can start with an arbitrary frame $\Kl{f_k}_{k\in\N}$ over $Y$ with the additional property 
    \begin{equation}\label{cond_norm_Z}
    	a_1 \norm{y}_Z^2 
    	\leq \sum_{k=1}^\infty \ak^2  \abs{\spr{y,f_k}_Y}^2 
    	\leq a_2 \norm{y}_Z^2
    	\,,
    	\qquad
    	\forall \, y \in Y \,,
	\end{equation}
for coefficients $0 \neq \ak \in \R$ and some constants $a_1,a_2 > 0$. Then, one defines
    \begin{equation*}
        e_k \coloneqq \ak A^* f_k \,, 
        \qquad
    \end{equation*}
which results in a frame $\Kl{e_k}_{k\in\N}$ over $X$ which satisfies \eqref{cond_frames_connected} with $\overline{\sigma_k} = 1/\ak$ \cite{Hubmer_Ramlau_2021_01}. In case that $Z$ and $Y$ are Sobolev spaces, frames $\Kl{f_k}_{k\in\N}$ satisfying \eqref{cond_norm_Z} can often be found (e.g., orthonormal wavelets \cite{Daubechies_1992}), which has resulted in FDs of the classic Radon transform \cite{Hubmer_Ramlau_2021_01,Hubmer_Ramlau_Weissinger_2022}. On the other hand, while the Funk-Radon transform satisfies a stability property of the form \eqref{cond_A_stability}, see Theorem~\ref{thm:sobolev} below, frames which satisfy \eqref{cond_norm_Z} are more difficult to find. The standard candidate would be spherical harmonics, which, however, already are the eigenfunctions of the Funk-Radon transform, and thus offer no further insight.

Hence, in this paper we consider a different approach for constructing FDs, which was originally outlined in \cite{Hubmer_Ramlau_2021_01}. This approach is still based on the stability property \eqref{cond_A_stability}, but instead of \eqref{cond_norm_Z} it only requires that
    \begin{equation}\label{cond_Y_Z}
		\norm{y}_Y \leq \norm{y}_Z  \,, 
		\qquad \forall \, y \in Z \,,
	\end{equation} 
that $Z \subseteq Y$ is a dense subspace of $Y$, and that the frame functions $f_k$ are elements of $Z$, i.e.,  $\norm{f_k}_Z < \infty$. In this case, one can build an alternative FD of $A$ similar to \eqref{frame_dec_A}, which can then be used to compute the (unique) solution of the linear operator equation $Ax = y$ for any $y$ in the range $\mathcal R(A)$ of $A$, and to develop stable reconstruction approaches in case of noisy data $\yd$.
 
The aim of this paper is to show that the above approach is applicable to the Funk-Radon transform. In particular, we construct an FD using trigonometric functions, which have the advantage of their fast computation outperforming spherical harmonics, cf.\ \cite{MilQue22,Yee80}. For this, we first review some background on frames and FDs in Section~\ref{sect_background}. Then, in Section~\ref{sect_Funk_Radon} we show that all required properties are satisfied for the Funk-Radon transform with a suitable choice of the functions $f_k$. Furthermore, we provide explicit expressions for the frame functions $e_k$, leading to an FD and a corresponding reconstruction formula. Finally, in Section~\ref{sect_numerics} we consider the efficient implementation of our derived FD and evaluate its reconstruction quality on a number of numerical examples.

% % % % % % % % % % % % % % % % % %
% Section - Frame Decompositions  %
% % % % % % % % % % % % % % % % % %
\section{Background on Frames and Frame Decompositions}\label{sect_background}

In this section, we review some background on frames and FDs, collected from the seminal works \cite{Christensen_2016,Daubechies_1992} and the recent article \cite{Hubmer_Ramlau_2021_01}, respectively.

\begin{definition}\label{def_frame}
A sequence $\{e_k\}_{k \in \N}$ in a Hilbert space $X$ is called a \emph{frame} over $X$, if and only if there exist \emph{frame bounds} $0< B_1,B_2 \in \R$ such that there holds
	\begin{equation}\label{eq_framedef}
    	B_1 \norm{x}_X^2 \leq \sum\limits_{k=1}^\infty \abs{\spr{x,e_k}_X}^2 \leq B_2 \norm{x}_X^2 \,,
        \qquad
        \forall \, x \in X \,.
	\end{equation}
\end{definition}

For a given frame $\{e_k\}_{k\in \N}$, one can consider the \emph{frame (analysis) operator} $F$ and its adjoint \emph{(synthesis)} operator $F^*$, which are given by
    \begin{equation}\label{def_F_Fad}
    \begin{split}
        &F \, \colon \, X \to \ell_2(\N) \,, \qquad
        x \mapsto \kl{\spr{x,e_k}_X}_{k \in \N} \,,
        \\ 
        &F^* \, \colon \, \ell_2(\N) \to X \,, \qquad
        \kl{a_k}_{k\in\N} \mapsto \sum\limits_{k=1}^\infty a_k e_k  \,.
    \end{split}	
    \end{equation}
Due to \eqref{eq_framedef} there holds $\sqrt{B_1} \leq \norm{F} = \norm{F^*} \leq \sqrt{B_2}$. Furthermore, one can define
    \begin{equation}\label{Def_S}
        S x \coloneqq F^*F x = \sum\limits_{k=1}^\infty \spr{x,e_k}_X e_k \,,
        \qquad
        \text{and}
        \qquad
        \ekt \coloneqq S^{-1} e_k \,.
    \end{equation}
Since $S$ is continuously invertible with $B_1 I \leq S \leq B_2 I$, the functions $\ekt$ are well-defined, and the set $\{\tilde{e}_k\}_{k\in \N}$ forms a frame over $X$ with frame bounds $B_2^{-1},B_1^{-1}$ called the \emph{dual frame} of $\{e_k\}_{k \in \N}$. Furthermore, it can be shown that 
    \begin{equation}\label{eq_frame_rec}
        x = \sum\limits_{k=1}^\infty \spr{x, \tilde{e}_k}_X e_k
        = \sum\limits_{k=1}^\infty \spr{x, e_k}_X \tilde{e}_k \,,
        \qquad
        \forall \, x \in X \,.
    \end{equation}
In general, this decomposition is not unique, which is a key difference between frames and bases, but it can be understood as the ``most economical'' one \cite{Daubechies_1992}.

Next, we consider FDs. For this, we use the following

\begin{assumption}\label{ass_main_III}
The operator $A \colon X \to Y$ {between the Hilbert spaces $X,Y$} is bounded, linear, and satisfies condition \eqref{cond_A_stability} for some constants $c_1,c_2 > 0$, where the Hilbert space $Z \subseteq Y$ is a dense subspace of $Y$ satisfying \eqref{cond_Y_Z}. Furthermore, the set $\Kl{f_k}_{k\in\N}$ forms a frame over $Y$ with frame bounds $C_1,C_2>0$, and $\{\fkt\}_{k\in\N}$ denotes the dual frame of $\Kl{f_k}_{k\in\N}$. Moreover, the functions $f_k$ are elements of $Z$, i.e., $\norm{f_k}_Z < \infty$, and $E\colon Z \to Y$, $z\mapsto z$, denotes the \emph{embedding operator}.
\end{assumption}

Now, the key idea for constructing an FD of $A$ is the suitable choice of a frame $\Kl{e_k}_{k\in \N}$ over $X$ based on the frame $\Kl{f_k}_{k\in \N}$, as outlined in

\begin{proposition}[{\cite[Lem.\ 4.5]{Hubmer_Ramlau_2021_01}}]\label{lem_ek_L_frame}
Let $A \colon X \to Y$ and let Assumption~\ref{ass_main_III} hold. Then the set $\{e_k\}_{k\in \N}$, where the functions $e_k$ are defined as
	\begin{equation}\label{def_ek_L}
		e_k \coloneqq A^*L f_k \,,
        \qquad
        \text{where}
        \qquad
        L \coloneqq (E E^*)^{-1/2} \,,
	\end{equation}
form a frame over $X$ with frame bounds $B_1= c_1^2\, C_1$ and $B_2 = c_2^2 \,C_2$, where $C_1$ and $C_2$ are the frame bounds of $\{f_k\}_{k\in\N}$, and $c_1$ and $c_2$ are as in Assumption~\ref{ass_main_III}.
\end{proposition}

The choice \eqref{def_ek_L} for the functions $e_k$ leads us to the following

\begin{proposition}[{\cite[Lem.\ 4.6]{Hubmer_Ramlau_2021_01}}]\label{lem_dec_LA}
Let $A \colon X \to Y$, let Assumption~\ref{ass_main_III} hold, and let the functions $e_k$ be defined as in \eqref{def_ek_L}. Then for all $x \in X$ there holds
	\begin{equation}\label{def_A_LA}
        L A x = \sum_{k=1}^\infty \spr{x,e_k}_X \fkt \,,
        \qquad
        \text{and}
        \qquad
        A x = L^{-1} \kl{\sum_{k=1}^\infty \spr{x,e_k}_X  \fkt} \,,
	\end{equation}
where the second equality uses the fact that $L^{-1} = (EE^*)^{1/2}$ is continuous.
\end{proposition}

Next, we consider the solution of the linear operator equation $Ax=y$ using the FD of $A$ from above. Summarizing results from Lemma~4.7, Theorem~4.8, and Remark~4.2 in \cite{Hubmer_Ramlau_2021_01}, which are essentially consequences of \eqref{def_A_LA}, we obtain

\begin{theorem}\label{thm_main_IIIII}
Let $A\colon X \to Y$ be a bounded linear operator, let Assumption~\ref{ass_main_III} hold, and let the functions $e_k$ be defined as in \eqref{def_ek_L}. Then for any $y \in \mathcal R(A) \subseteq Z$,
    \begin{equation}\label{def_ADt}
		\ADD y \coloneqq \sum_{k=1}^\infty \spr{Ly ,f_k}_Y \ekt \,,
	\end{equation}
is the well-defined, unique solution of $Ax=y$, and	$\norm{\ADD y}_X \leq \sqrt{C_2/B_1}  \norm{L y}_Y$.
\end{theorem}

% % % % % % % % % % % % % % % % % % % % % % % % % % % % % % % % 
% Section - Frame Decompositions of the Funk-Radon Transform  %
% % % % % % % % % % % % % % % % % % % % % % % % % % % % % % % %
\section{Frame Decompositions of the Funk-Radon Transform}\label{sect_Funk_Radon}

% Subsection - Eigenvalue Decompositions of the Funk-Radon Transform

The eigenvalue decomposition of the Funk-Radon transform \eqref{eq:fr}, which is also an FD, is due to \cite{Min05}. Denoting by $P_\ell$ the $\ell$-th Legendre polynomial, we have
    \begin{equation} \label{eq:fr-svd}
        RY_\ell^m
        =
        P_\ell(0)\, Y_\ell^m
        =
        \begin{cases}
            \frac{(-1)^{\ell/2} (\ell-1)!!}{\ell!!}\, Y_\ell^m\,, & \ell\text{ even}\,,\\
            0\,, & \ell\text{ odd}\,,
        \end{cases}
    \end{equation}
where the eigenfunctions are the \emph{spherical harmonics} $Y_\ell^m$ \cite{Mue66} of degree $\ell\in\NN$ and order $m=-\ell,\dots,\ell$, which form an orthonormal basis of $L^2(\S^{2})$.
From its definition in \eqref{eq:fr}, we see that $Rf$ is even for any $f\colon \S^2\to\C$, i.e., $Rf(\zb\xi) = Rf(-\zb\xi)$ for all $\zb\xi\in\S^2$. Conversely, $Rf$ vanishes for odd functions $f(\zb\xi) = -f(-\zb\xi)$, so we can expect to recover only even functions $f$ from their Funk-Radon transform.

The spherical \emph{Sobolev space} $H^s(\S^{2})$, $s\in\R$, can be defined as the completion of $C^\infty(\S^{2})$ with respect to the norm  \cite{AtHa12}
    \begin{equation} \label{eq:Hs-norm}
        \norm{f}_{H^s(\S^{2})}
        \coloneqq
        \sum_{\ell=0}^{\infty} \sum_{m=-\ell}^{\ell}
        (\ell+\tfrac{1}{2})^{2s} |\langle f,Y_\ell^m\rangle_{L^2(\S^{2})} |^2 \,,
    \end{equation}
where $\langle f,g\rangle_{L^2(\S^{2})} \coloneqq \int_{\S^2} f(\zb\xi)g(\zb\xi) \d\zb\xi$,
and we denote by $H^s_{\mathrm{even}}(\S^2)$ its restriction to even functions, which is the span of spherical harmonics $Y_\ell^m$ with even degree $\ell\in 2\NN$. The Sobolev spaces are nested, i.e., $H^s(\S^2) \subsetneq H^r(\S^2)$ whenever $s>r$, and we have $H^0(\S^2) = L^2(\S^2)$.

\begin{theorem} \label{thm:sobolev}
Let $s\ge0$. The Funk-Radon transform $R$ defined in \eqref{eq:fr} extends to a continuous and bijective operator from $X=H^s_{\mathrm{even}}(\S^{2})$ to $Z=H^{s+1/2}_{\mathrm{even}}(\S^{2})$ that satisfies \eqref{cond_A_stability} with the bounds $c_1 = {\sqrt{1/2}}$ and $c_2 = {\sqrt{2/\pi}}$.
Furthermore, it also extends to a continuous and self-adjoint operator from $H^s_{\mathrm{even}}(\S^{2})$ to $H^{s}_{\mathrm{even}}(\S^{2})$.
\end{theorem}
\begin{proof}
The bijectivity of $R\colon X\to Y$ is due to \cite[§\,4]{Str81}.
From Theorem 3.13 in \cite{quellmalzdiss}, we know that $c_1$ and $c_2$ in \eqref{cond_A_stability} are characterized by
\begin{equation} \label{eq:Pl_bound}
    c_1 \left(\ell + \tfrac12 \right)^{-\frac12} \leq \frac{(\ell-1)!!}{\ell !!}=\abs{P_\ell(0)} \leq c_2 \left(\ell + \tfrac12 \right)^{-\frac12}
    \,,\qquad \forall \, \ell \in 2\N_0\,.
\end{equation}
Analogously to the proof of Lemma 3.2 in \cite{HiPoQu18}, we can see that the sequence $2\NN\ni \ell\mapsto (\ell+1/2)^{1/2}\, (\ell-1)!!/\ell !!$ is increasing and converges to $2/\pi$ for $\ell\to\infty$.
Therefore, it is bounded from below by its value $1/2$ for $\ell=0$ and from above by its limit $2/\pi$. Furthermore, $R\colon X\to X$ is self-adjoint since its eigenvalues $P_\ell(0),$ cf.\ \eqref{eq:fr-svd}, are real.
\qed
\end{proof}

% Subsection - A Frame Decomposition of the Funk-Radon Transform

In the following, we describe a trigonometric basis on the sphere that allows us to obtain a novel FD of the Funk-Radon transform $R$. We start with the \emph{spherical coordinate transform} 
    \begin{equation} \label{eq:s2-coord}
        \phi(\lambda,\theta)
        \coloneqq
        (\cos\lambda\, \sin\theta,\, \sin\lambda\,\sin\theta,\, \cos\theta)\,,
        \qquad \forall\, \lambda\in[0,2\pi)\,,\ \theta\in[0,\pi]\,,
    \end{equation}
which is one-to-one except for $\theta\in\{0,\pi\}$ corresponding to the north and south pole.
We assume $\lambda$ to be $2\pi$-periodic, and define the trigonometric basis functions
    \begin{equation} \label{eq:basis-s2}
        b_{n,k}\colon \S^2\to\C\,,\ 
        b_{n,k}(\phi(\lambda,\theta))
        \coloneqq
        \frac{\e^{\i n \lambda}\, \sin(k\theta)}{\pi \sqrt{\sin\theta}}
        \,,\qquad \forall\, n\in\Z\,,\ k\in\N\,,
    \end{equation}
which are well-defined and continuous on $\S^2$ since $b_{n,k}$ vanishes for $\theta\to0$ and $\theta\to\pi$.
Trigonometric bases on $\S^2$ bear the advantage of their simple and fast computation \cite{Yee80}, and form the foundation of the double Fourier sphere method \cite{MilQue22,WilTowWri17}. Note that our functions \eqref{eq:basis-s2} slightly differ from the ones in \cite{MilQue22} as we take $\sin(k\theta)$ in order to avoid singularities at the poles.

\begin{lemma} \label{thm:bnk-L2}
  The sequence $\{b_{n,k}\}_{n\in\Z, k\in\N}$ forms an orthonormal basis of $L^2(\S^2)$.
\end{lemma}
\begin{proof}
Let $n,n'\in\Z$ and $k,k'\in\N$.
Since the integral on $\S^2$ in spherical coordinates \eqref{eq:s2-coord} reads as 
$
\sin(\theta) \d\theta \d\lambda,
$
we have
    \begin{equation}
        \langle b_{n,k}, b_{n',k'} \rangle_{L^2(\S^2)}
        =
        \frac{1}{\pi^2}
        \int_{0}^{\pi} \int_{0}^{2\pi} \e^{\i (n - n') \lambda}\, {\sin(k\theta)}\, {\sin(k'\theta)} \d\lambda \d\theta
        =
        \delta_{n,n'} \delta_{k,k'}\,,
    \end{equation}
which shows the orthonormality. The completeness follows from the completeness of $\{\e^{\i n \cdot}\}_{n\in\Z}$ in $L^2([0,2\pi])$ and of $\{\sin(k \cdot)\}_{k\in\N}$ in $L^2([0,\pi])$.
\qed
\end{proof}

For $\zb\xi = \phi(\lambda,\theta)\in\S^2$, its antipodal point is $-\zb\xi = \phi(\pi+\lambda,\pi-\theta)$. Hence, we obtain the symmetry relation
$
b_{n,k}(-\zb\xi)
=
(-1)^{n+k+1} b_{n,k}(\zb\xi)\,,
$
which implies that the sequence $\{b_{n,k}\}_{n\in\Z, k\in\N, n+k\ \mathrm{odd}}$ is an orthonormal basis of $L^2_{\mathrm{even}}(\S^2)$.

\begin{lemma} \label{thm:bnk-H1}
  The basis functions $b_{n,k}$, $n\in\Z,$ $k\in\N$ with $n+k$ odd are well-defined elements of $H^1_{\mathrm{even}}(\S^2)$.
  In particular, we also have $b_{n,k} \in H^{1/2}_{\mathrm{even}}(\S^2)$.
\end{lemma}
\begin{proof}
  By \cite[§\ 5.2]{Que20}, the $H^1(\S^2)$ Sobolev norm of $f\in C^1(\S^2)$ can be written as
  \begin{equation} \label{eq:H1}
    \norm{f}_{H^1(\S^2)}^2
    =
    \sum_{i=1}^3 \norm{[\nabla_{\mathbb{S}^2}f]_i}_{L^2(\S^2)}^2
    +
    \tfrac14 \norm{f}_{L^2(\S^2)}^2\,,
  \end{equation}
  where $[\nabla_{\mathbb{S}^2}f]_i$ denotes the $i$-th coordinate of the surface gradient given in spherical coordinates by
  \begin{equation}
  \nabla_{\mathbb{S}^2}f(\phi(\lambda,\theta))
  =
  \phi(\lambda,\tfrac\pi2+\theta)\, \partial_\theta f(\phi(\lambda,\theta))
  +
  \phi(\tfrac\pi2+\lambda,\tfrac\pi2)\, \tfrac{1}{\sin\theta}\, \partial_\lambda f(\phi(\lambda,\theta))\,.
  \end{equation}
  Let $n\in\Z$ and $k\in\N$.
  For $f=b_{n,k}$, we have 
  \begin{equation}
  \partial_\theta b_{n,k}(\phi(\lambda,\theta))
  =
  \frac{\e^{\i n\lambda}}{\pi} \left( \frac{k \cos(k\theta)}{\sqrt{\sin\theta}} -\frac{\cos(\theta)\, \sin(k\theta)}{2\, \sin^{3/2}\theta } \right)
  \end{equation}
  and
  \begin{equation}
  \partial_\lambda b_{n,k}(\phi(\lambda,\theta))
  =
  \i n\, \frac{\e^{\i n\lambda}}{\pi}\, \frac{\sin(k\theta)}{\sqrt{\sin\theta}}\,.
  \end{equation}
  Employing the facts that sine and cosine functions, in particular the components of $\phi$, are bounded by one and that $\abs{a+b}^2\leq 2(\abs{a}^2+\abs{b}^2)$ for $a,b\in\R$, we obtain
  \begin{align}
  \abs{[\nabla_{\mathbb{S}^2} b_{n,k}(\phi(\lambda,\theta))]_i}^2
  &\le
  2 \abs{ \frac{k \cos(k\theta)}{\pi\sqrt{\sin\theta}} -\frac{\cos(\theta)\, \sin(k\theta)}{2 \pi\, \sin^{3/2}\theta} }^2
  +
  2 \frac{n^2 \sin^2(k\theta)}{\pi^2\sin^3\theta}
  \\
  &\le
  \frac{4 k^2}{\pi^2\sin\theta} +\frac{ \sin^2(k\theta)}{\pi^2\, \sin^3\theta} 
  +
  \frac{2n^2 \sin^2(k\theta)}{\pi^2\,\sin^3\theta}\,.
  \end{align}
  Hence, we have
  \begin{align}
    \norm{[\nabla_{\mathbb{S}^2}b_{n,k}]_i}_{L^2(\S^2)}^2
    &=
    \int_{0}^{\pi} \int_{0}^{2\pi} \abs{[\nabla_{\mathbb{S}^2}b_{n,k}]_i}^2\, \sin(\theta) \d\lambda\d\theta
    \\
    &\le
    \int_{0}^{\pi} \int_{0}^{2\pi} \left(\frac{4 k^2}{\pi^2} + (1+2n^2)\frac{ \sin^2(k\theta)}{\pi^2\sin^2\theta} \right) \d\lambda\d\theta
    \\
    &=
    8 k^2 + (1+2n^2) 2k\,,
  \end{align}
  where the last equality follows from the the integration formula \cite[Ex.\ 1.15]{PlPoStTa18} of the $(k-1)$th Fejér kernel.
  Finally, we conclude from \eqref{eq:H1} and Lemma \ref{thm:bnk-L2} that
  \begin{equation}
    \norm{b_{n,k}}_{H^1(\S^2)}^2
    =
    \sum_{i=1}^3 \norm{[\nabla_{\mathbb{S}^2}b_{n,k}]_i}_{L^2(\S^2)}^2
    +
    \tfrac14 \norm{b_{n,k}}_{L^2(\S^2)}^2
    \le
    6k(4 k + 1+2n^2) + \tfrac14
  \end{equation}
  is finite.
  The claim follows as $H^1(\S^2)$ is continuously embedded in $H^{1/2}(\S^2)$.
  \qed
\end{proof}

\begin{theorem} \label{thm:fr-fd}
  Let $E\colon H^{1/2}_\mathrm{even}(\S^2) \to L^2_{\mathrm{even}}(\S^2)$ denote the embedding operator,
  and set $L\coloneqq (EE^*)^{-1/2}$.
  Then
  \begin{equation}
    e_{n,k} \coloneqq 
    R L b_{n,k}
    \,,\qquad (n,k)\in J \coloneqq \{(n,k)\in\Z\times\N: n+k\text{ odd}\}\,,
  \end{equation}
  is a frame in $L^2_{\mathrm{even}}(\S^2)$,
  and for any $g\in H^{1/2}_{\mathrm{even}}(\S^2)$,
  the unique solution $f\in L^{2}_{\mathrm{even}}(\S^2)$ of the inversion problem of the Funk-Radon transform
  $ Rf = g$
  satisfies
  \begin{equation}\label{FD_inversion}
    f
    =
    R^\ddagger g
    \coloneqq
    \sum_{(n,k)\in J}
    \langle Lg, b_{n,k} \rangle_{L^2(\S^2)}
    \tilde e_{n,k}\,,
  \end{equation}
  where $\tilde e_{n,k}$ is the dual frame of $e_{n,k}$.
  It holds that
  $
	\norm{R^\ddagger g}_{L^2(\S^2)} \leq 2  \norm{L g}_{L^{2}(\S^2)}.
  $
\end{theorem}
\begin{proof}
  From Theorem~\ref{thm:sobolev}, Lemmas \ref{thm:bnk-L2} and \ref{thm:bnk-H1}, we see that Assumption \ref{ass_main_III} is satisfied with {$X=Y=L_{\text{even}}^2(\S^2)$ and $Z=H_{\text{even}}^{1/2}(\S^2)$}.
  The claim follows by Theorem~\ref{thm_main_IIIII}.
  \qed
\end{proof}

\begin{remark} \label{rem:L}
The definition \eqref{eq:Hs-norm} of the Sobolev spaces yields that
\begin{equation}
  E^* f
  =
  EE^* f
  =
  \sum_{\ell=0}^{\infty} \sum_{m=-\ell}^{\ell} (\ell+\tfrac12)^{-1} \langle f, Y_\ell^m \rangle_{L^2(\S^2)}\, Y_\ell^m
  \,,\qquad \forall f\in L^2(\S^2)\,.
\end{equation}
Hence, the self-adjoint operator $L$ from Theorem \ref{thm:fr-fd} is a multiplication operator with respect to the spherical harmonics, and we have for all $g\in H^{1/2}(\S^2)$ that
\begin{equation} \label{eq:L}
  L g
  =
  (EE^*)^{-1/2} g
  =
  \sum_{\ell=0}^{\infty} \sum_{m=-\ell}^{\ell} (\ell+\tfrac12)^{1/2} \langle g, Y_\ell^m \rangle_{L^2(\S^2)}\, Y_\ell^m
  \,.
\end{equation}
Denoting by $\Delta_{\mathbb{S}^2}$ the Laplace-Beltrami operator on $\S^2$, we obtain by its eigenvalue decomposition \cite[p.\ 121]{AtHa12} that
$
    Lg
    =
    (-\Delta_{\mathbb{S}^2} + 1/4)^{1/4} g
$
for
$
    g\in H^{1/2}(\S^2) .
$
Furthermore, combining \eqref{eq:fr-svd} and \eqref{eq:L}, it follows that
\begin{equation} \label{eq_RLg}
  R L g
  =
  \sum_{\ell=0}^{\infty} \sum_{m=-\ell}^{\ell} P_\ell(0)\, (\ell+\tfrac12)^{1/2} \langle g, Y_\ell^m \rangle_{L^2(\S^2)}\, Y_\ell^m
  \,,
\end{equation}
and since $\abs{P_\ell(0)}$ decays as $(\ell+\frac12)^{-1/2}$ by \eqref{eq:Pl_bound}, we obtain $RLg\in H^{1/2}(\S^2)$.
\end{remark}

% % % % % % % % % % % % % % % %
% Section - Numerical Results %
% % % % % % % % % % % % % % % %
\section{Numerical Results}\label{sect_numerics}

Next, we discuss the use of regularization for our FD of the Funk-Radon transform, outline the main steps of its implementation, and present numerical results. First, note that noisy data $g^\delta \coloneqq Rf+\delta$ does not necessarily belong to the space $H^{1/2}(\S^2)$, and thus $R^\ddagger g^\delta$ is not well-defined. This necessitates regularization, for which we consider stable approximations of $R^\ddagger g^\delta$ defined by 
\begin{equation*}
    f_\alpha^\delta 
    \coloneqq
    R^\ddagger U_\alpha g^\delta \,,
    \qquad
    \text{and}
    \qquad
    L U_\alpha g^\delta
    \coloneqq
    \sum_{\ell=0}^\infty\sum_{m=-\ell}^\ell \sigma_\ell h_\alpha(\sigma_\ell^2) \langle g^\delta,Y_\ell^m\rangle_{L^2(\S^2)}\, Y_\ell^m\,,
  \end{equation*}
where $\sigma_\ell \coloneqq (\ell+\frac{1}{2})^{-1/2}$
and $R^\ddagger$ is given in \eqref{FD_inversion}.
This amounts to a filtering of the coefficients $\sigma_k$ of $L$ via a suitable \emph{filter function} $h_\alpha$ approximating $s\mapsto 1/s$ as in the SVD case \cite{Engl_Hanke_Neubauer_1996}. In our numerical experiments, we investigate the positive influence of the Tikhonov filter functions $h_\alpha(s)=1/(\alpha+s)$ on the reconstruction quality. Note that the choice $h_\alpha(s)=1/\sqrt{s}$ implies $U_\alpha = L^{-1}$, which by Remark~\ref{rem:L} basically amounts to a smoothing operator; cf., e.g., \cite{Klann_Ramlau_2008}.

\begin{figure}
    \centering
    \includegraphics[width=0.32\textwidth,trim={.5cm 1.2cm .5cm 1.2cm},clip]{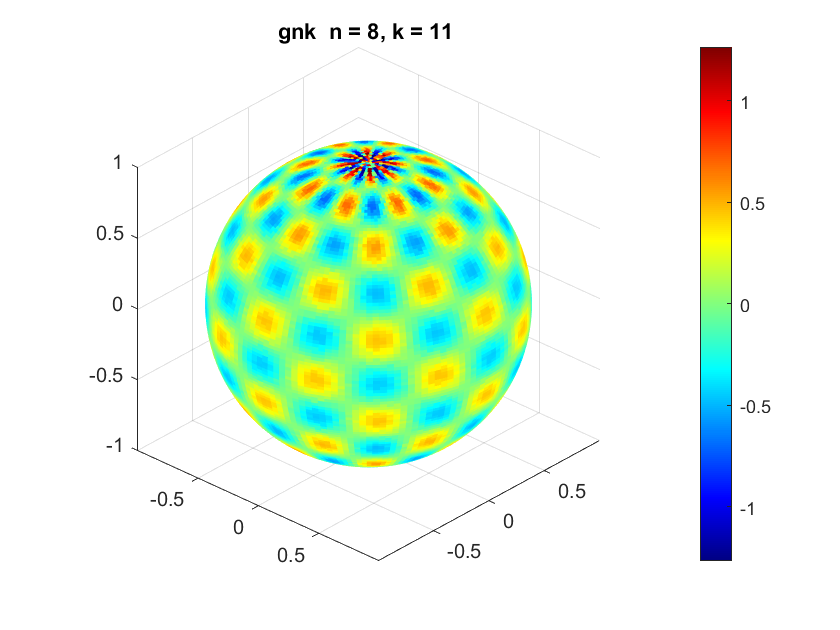}
    \includegraphics[width=0.32\textwidth,trim={.5cm 1.2cm .5cm 1.2cm},clip]{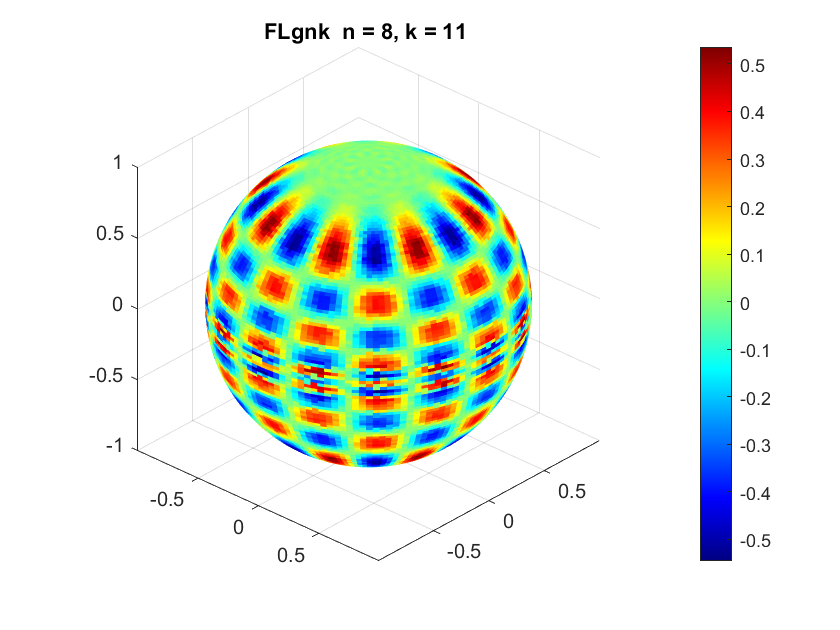}
    \includegraphics[width=0.32\textwidth,trim={.5cm 1.2cm .5cm 1.2cm},clip]{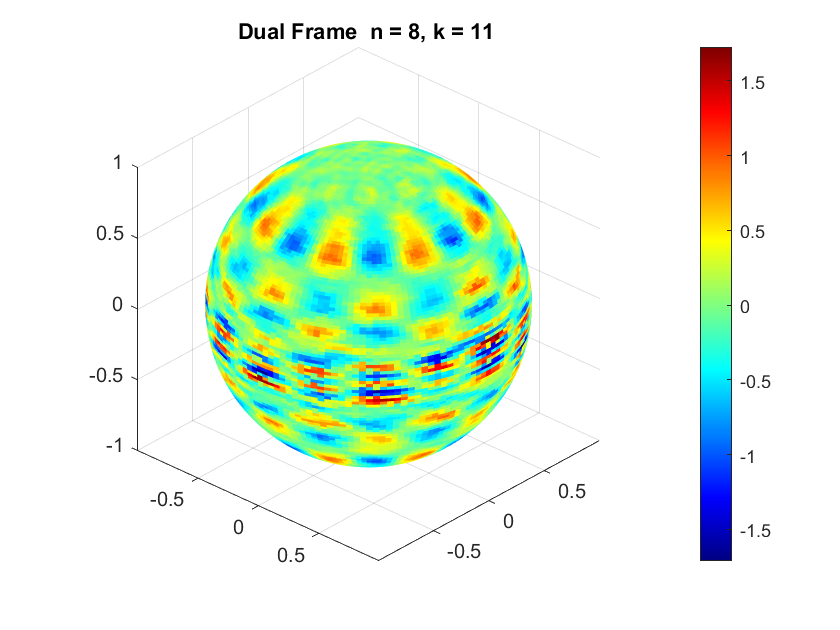}
    \caption{Ingredients of the FD: $b_{nk}$ (left), $e_{nk}$ (middle), $\Tilde{e}_{nk}$ (right) for $n=k=8$.}
    \label{fig:frames}
\end{figure}

For discretization of the problem, we approximate functions on the sphere with a Chebyshev-type quadrature, i.e., quadrature points such that all weights are equal; see \cite{GrPo10}. In our computations, we use quadrature points (spherical design) being exact up to degree 200 taken from \cite{Gr_points}. All computations involving spherical harmonics are done using the NFSFT (Non-equispaced fast spherical Fourier transform) software \cite{KeKuPo09,KuPo03}. Note that the dual frame functions of an FD can be pre-computed and stored, such that the computational effort of computing $R^\ddagger g$ according to \eqref{FD_inversion} for any new measurement $g$ amounts to one application of the operator $L$ (or $L U_\alpha$ in the noisy case), computation of the inner products $\langle Lg,b_{n,k}\rangle_{L_\text{even}^2(\S^2)}$ and a summation over these. As starting point for the computation, we only use a finite set of frame functions $\{b_{n,k}: (n,k)\in J , |n|\leq N , k\leq N\}$ for some $N\in\N.$ The frame functions $e_{n,k}=RLb_{n,k}$ are evaluated at the quadrature nodes via the spherical harmonic decomposition \eqref{eq_RLg} up to degree $\ell\leq 100$. The dual frame functions $\tilde{f}_{n,k}$ are computed using the matrix representation of the linear operator $S$ with respect to the basis $b_{n,k}$, see also \cite[Chap.\ 5]{Hubmer_Ramlau_Weissinger_2022}. Note that the inversion of this matrix may itself be an ill-conditioned problem requiring regularization. Examples of the involved frame functions are depicted in Figure~\ref{fig:frames}. All computations are performed using Matlab R2022b.

% Subsection - Numerical Experiments

\begin{figure}[ht!]
    \centering
    \begin{tabular}{c}
    \subfloat[\centering The test function $f$]{{\includegraphics[width=0.3\textwidth,trim={10 20 10 20},clip]{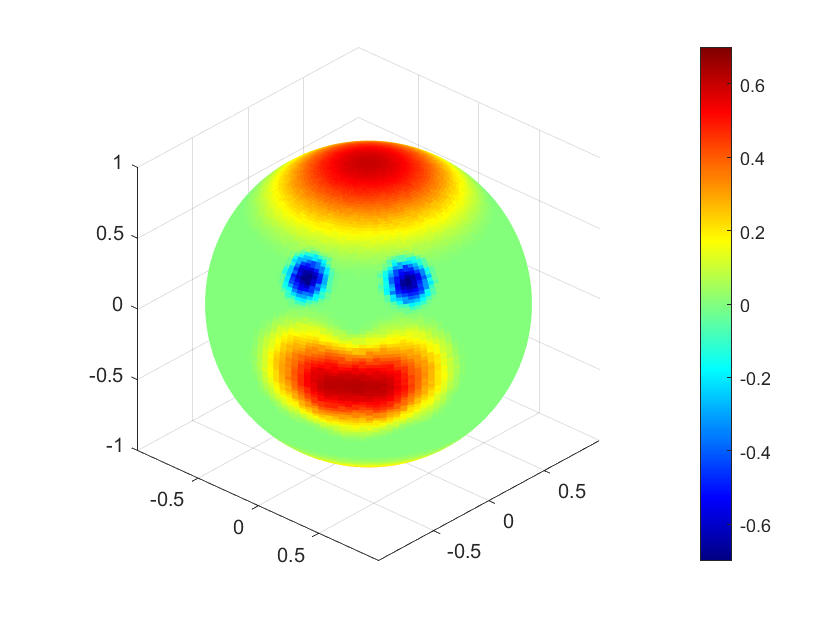} } \label{fig:ground_truth}}%
    \subfloat[\centering $R^\ddagger\,g$ for $N=25$]{{\includegraphics[width=0.3\textwidth,trim={10 20 10 20},clip]{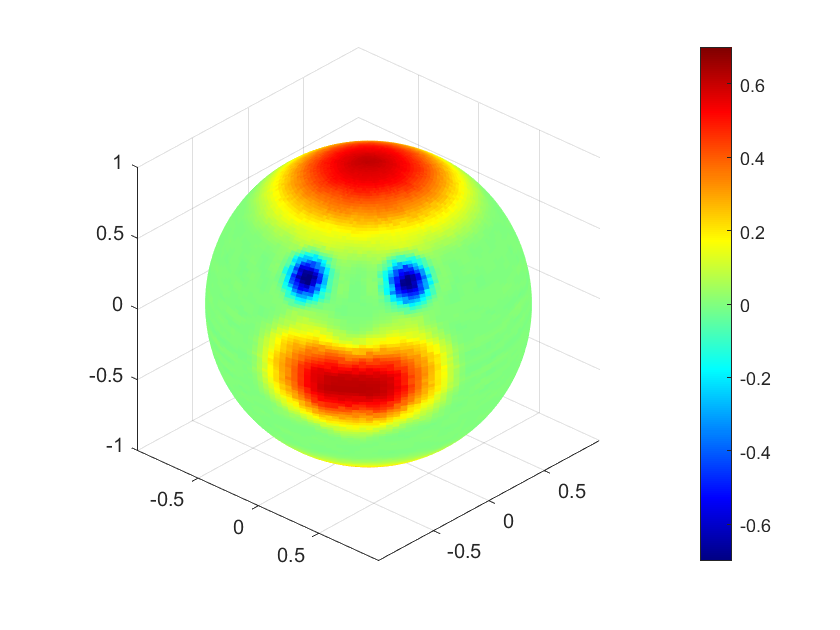} } \label{fig:REC_Tcheby_N25_K25_exact}}%
    \subfloat[\centering $|f-R^\ddagger\,g|$ for $N=25$]{{\includegraphics[width=0.3\textwidth,trim={10 20 10 20},clip]{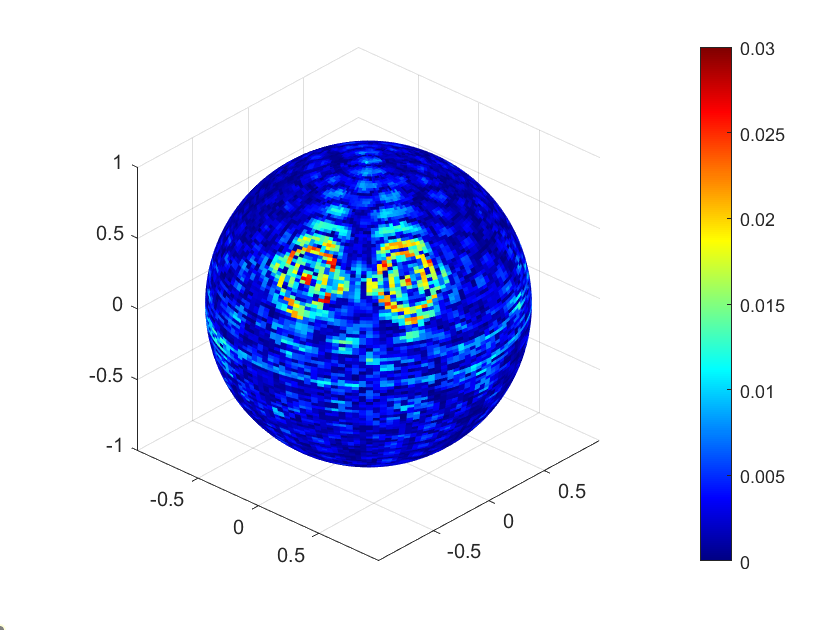} } \label{fig:ERR_Tcheby_N25_K25_exact}}%
    
    \\
    \subfloat[\centering Data $g=R\,f$]{{\includegraphics[width=0.3\textwidth,trim={10 20 10 20},clip]{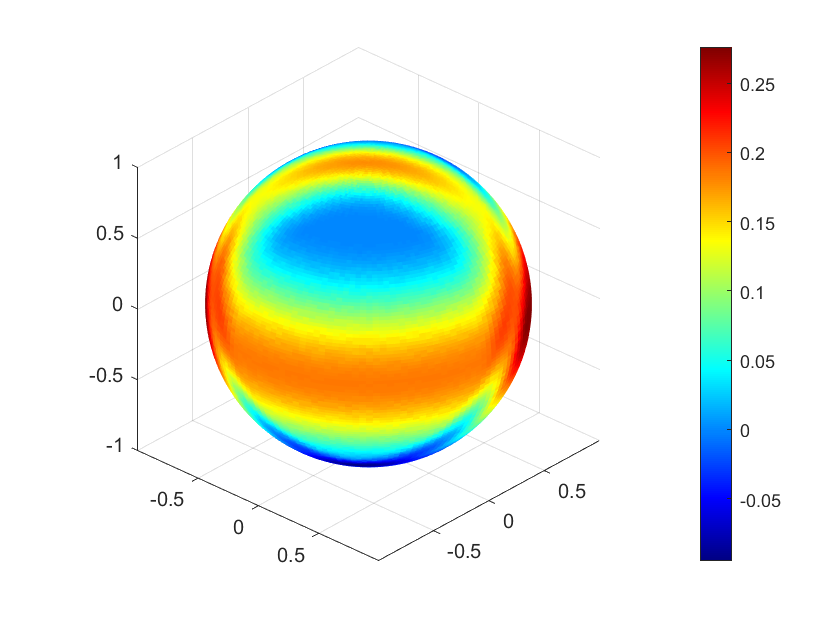} } \label{fig:unperturbeddata}}%
    \subfloat[\centering $R^\ddagger\,g$ for $N=40$]{{\includegraphics[width=0.3\textwidth,trim={10 20 10 20},clip]{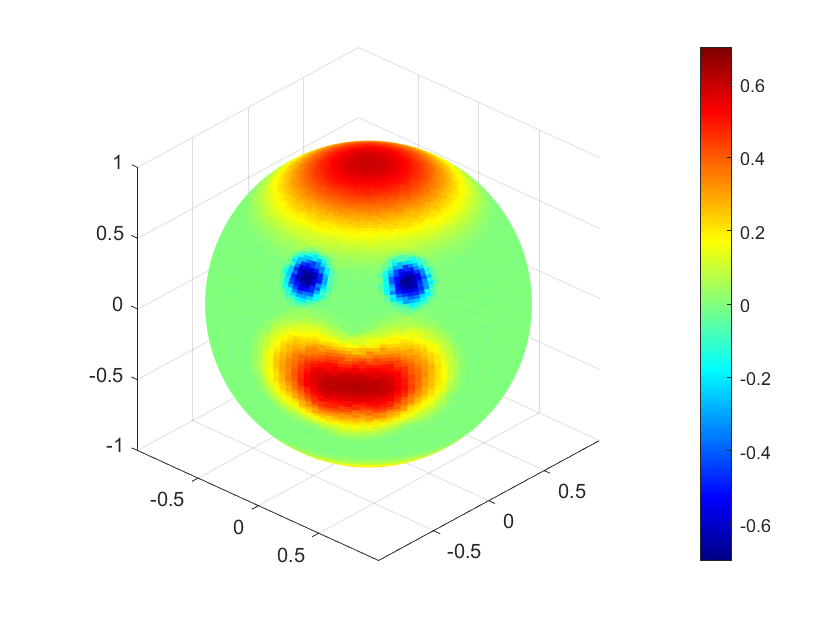} } \label{fig:REC_Tcheby_N40_K40_exact}}%
    \subfloat[\centering $|f-R^\ddagger\,g|$ for $N=40$]{{\includegraphics[width=0.3\textwidth,trim={10 20 10 20},clip]{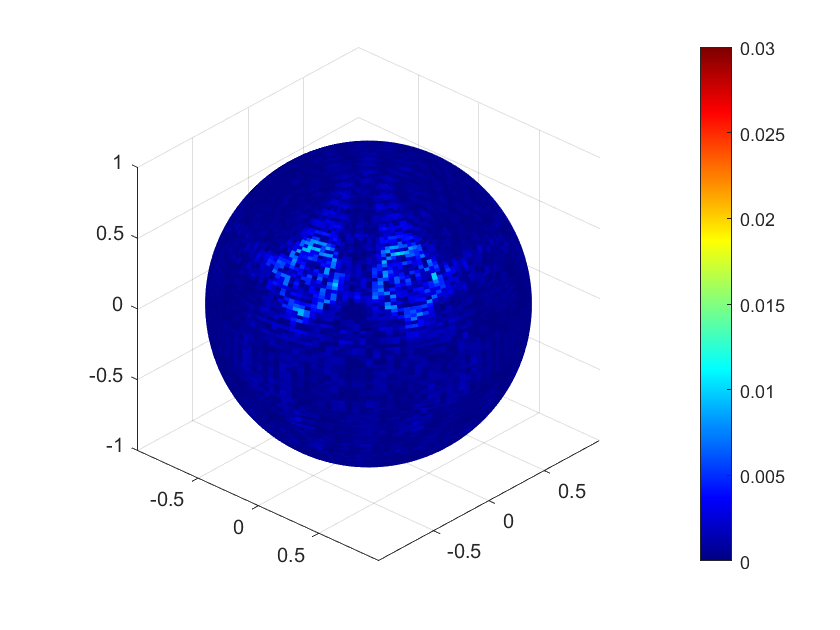} } \label{fig:ERR_Tcheby_N40_K40_exact}}%
    
    \end{tabular}
    \caption{Reconstruction evaluation for the Chebyshev-type quadrature for different numbers of dual-frame functions used with exact data. }
    \label{fig:exact_whole}
\end{figure}

\begin{figure}[ht!]
    \centering
    \begin{tabular}{c}
    \subfloat[\centering The test function $f$]{{\includegraphics[width=0.3\textwidth,trim={10 20 10 20},clip]{Figures/ground_truth.png} } \label{fig:ground_truth2}}%
    \subfloat[\centering $R^\ddagger U_\alpha\,g^\delta$ for $N=25$, $\alpha=0.064$]{{\includegraphics[width=0.3\textwidth,trim={10 20 10 20},clip]{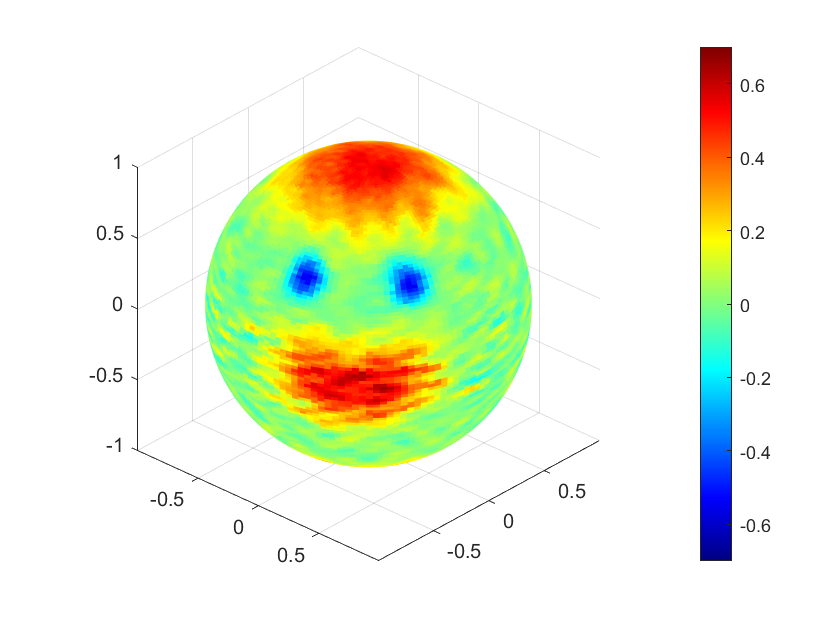} } \label{fig:REC_Tcheby_N25_K25_noise20}}%
    \subfloat[\centering $|f-R^\ddagger U_\alpha\,g^\delta|$ for $N=25$]{{\includegraphics[width=0.3\textwidth,trim={10 20 10 20},clip]{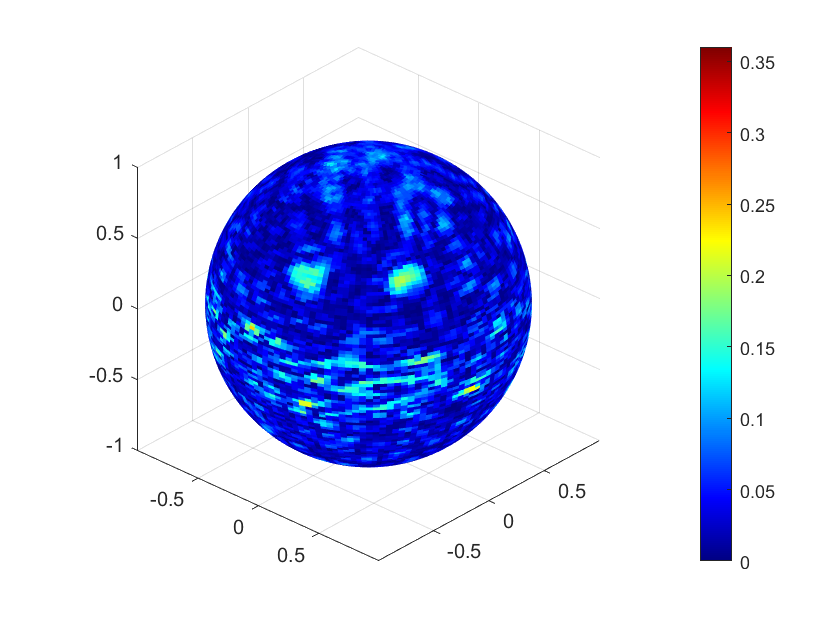} } \label{fig:ERR_Tcheby_N25_K25_noise20}}%
    
    \\
    \subfloat[\centering Noisy data $g^\delta=R\,f+\delta$]{{\includegraphics[width=0.3\textwidth,trim={10 20 10 20},clip]{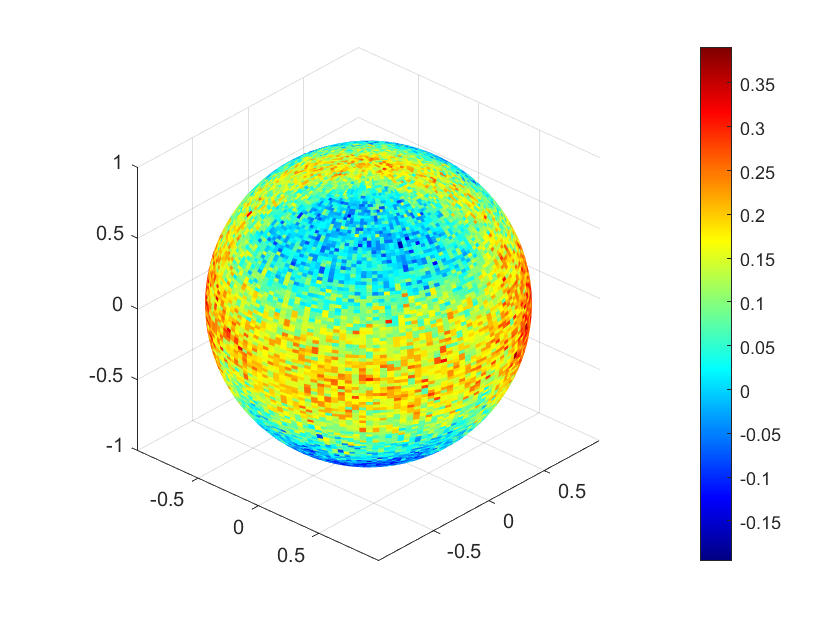} } \label{fig:noisydata20}}%
    \subfloat[\centering $R^\ddagger U_\alpha\,g^\delta$ for $N=40,$ $\alpha=0.14$]{{\includegraphics[width=0.3\textwidth,trim={10 20 10 20},clip]{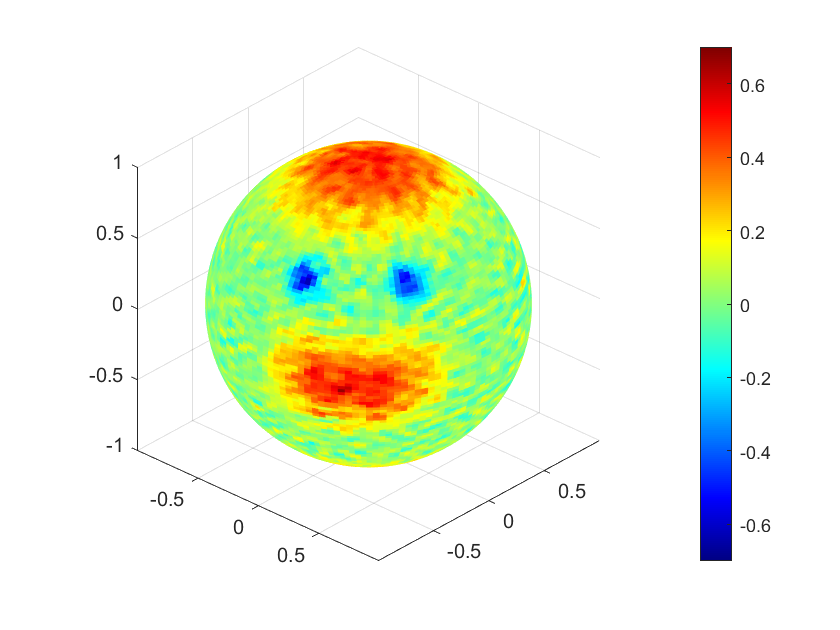} } \label{fig:REC_Tcheby_N40_K40_noise20}}%
    \subfloat[\centering $|f-R^\ddagger U_\alpha\,g^\delta|$ for $N=40$]{{\includegraphics[width=0.3\textwidth,trim={10 20 10 20},clip]{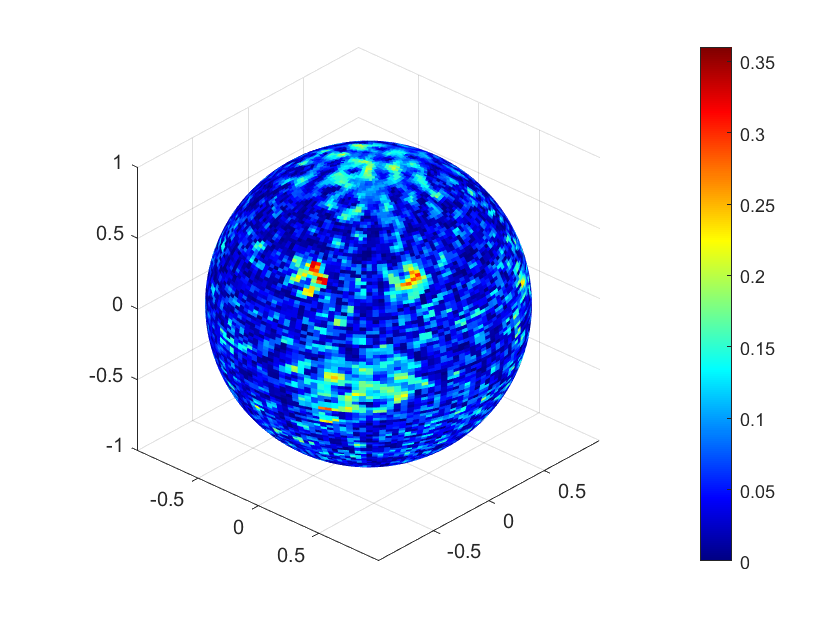} } \label{fig:ERR_Tcheby_N40_K40_noise20}}%
    
    \end{tabular}
    \caption{Reconstruction evaluation for the Chebyshev-type quadrature for different numbers of dual-frame functions used with Gaussian noise $\delta$, noise level $20\% $. }
    \label{fig:noise20_whole}
\end{figure}

The test function used in our numerical experiments is a linear combination of radially symmetric, quadratic splines, whose Funk-Radon transform is computed explicitly \cite[Lem.\ 4.1]{HiQu15}, to prevent inverse crimes. Our quality measure is the relative reconstruction error, i.e., $\|f-R^\ddagger g\|_{L^2(\S^2)}/\|f\|_{L^2(\S^2)}.$ In Figure~\ref{fig:exact_whole}, we see that increasing the number of used frames highly improves reconstruction quality, reducing the relative reconstruction error from $0.023$ for $N=25$ to $0.006$ for $N=40$. For noisy data shown in Figure~\ref{fig:noise20_whole}, the regularization parameter $\alpha$ is chosen such that the relative reconstruction error is minimized. In the case of $N=25$ and noise level $20\%$, the error for the non-regularized solution (i.e., $\alpha=0$) is $0.269,$ while the error for the regularized solution with parameter $\alpha=0.076$ reduces to $0.222.$ However, we see that the increment of the number of frame functions actually results in a loss of reconstruction quality to an error value of $0.332$ (optimally regularized). This can be explained by the regularization effect of the truncation itself: more frame functions result in a less stable reconstruction, but yield a higher accuracy in case of exact data. Note that all specific error values in the noisy case are insignificantly varying for the specific realization of the randomly generated Gaussian noise $\delta$.

% % % % % % % %
% Conclusion  %
% % % % % % % %
\section{Conclusion} 

In this paper, we derived a novel frame decomposition of the Funk-Radon transform utilizing trigonometric basis functions $b_{n,k}$ on the unit sphere and suitable embedding operators in Sobolev spaces. This decomposition does not involve the spherical harmonics and leads to an explicit inversion formula for the Funk-Radon transform. In our numerical examples, we obtained promising reconstruction results even in the case of very large noise by including regularization. While the regularization itself currently uses a spherical harmonics expansion of the operator $L$, in our future work we aim to apply other forms of regularization avoiding the computationally expensive spherical harmonics entirely.

% % % % % % % % % % %
% Acknowledgements  %
% % % % % % % % % % %
\subsubsection{Acknowledgements} 

This work was funded by the Austrian Science Fund (FWF): project F6805-N36 (SH) and the German Research Foundation (DFG): project 495365311 (MQ) within the SFB F68: ``Tomography Across the Scales''. LW is partially supported by the State of Upper Austria.
%The funding of this work is not disclosed in this anonymously written version for review.

% % % % % % % % %
% Bibliography  %
% % % % % % % % %

\bibliographystyle{splncs04}
\bibliography{mybib,ref}

\end{document}